\documentclass[10pt]{amsart}

\usepackage{amssymb,amsmath,enumerate} 
\usepackage{graphicx}
\usepackage{caption}
\usepackage{subcaption}
\usepackage{tikz}
\usepackage{booktabs}
\usepackage{adjustbox}

\newtheorem{theorem}{Theorem}[section]
\newtheorem{lemma}[theorem]{Lemma}

\newtheorem{conjecture}[theorem]{Conjecture}
\theoremstyle{definition}

\theoremstyle{remark}

\theoremstyle{proposition}
\newtheorem{proposition}[theorem]{Proposition}
\numberwithin{equation}{section}

\newcommand{\Area}{{\rm area}\,}
\newcommand{\be}{\begin{equation}}
\newcommand{\ee}{\end{equation}}

\renewcommand{\d}[1]{\,\mathrm{d}{#1}}

\title{Growth of hyperbolic cells} 
\author[P.Chakraborty]{Pritha Chakraborty}
\address{Department of Mathematics and Statistics, Texas Tech
University, Lubbock, TX 79409} \email{pritha.chakraborty@ttu.edu}

\begin{document}%


\begin{abstract}%
We consider a hyperbolic polygon in the unit disk $\{z:\, |z|<1\}$ with all its vertices
on the unit circle $\{z:\, |z|=1\}$ and a growth process of such
polygons when each $n$-gon generates an $n(n-1)$-gon by inverting itself across
all of its sides. In this paper, we prove some general monotonicity results of inversion for 
convex hyperbolic $n$-gons and 
solve an extremal problem that, among all convex hyperbolic $4$-gons containing the origin,
the inverted side length of the longest side of the given hyperbolic $4$-gon 
is minimal for the regular hyperbolic $4$-gon.
\end{abstract}%

\maketitle%

\section{Introduction}

We consider the Poincar\'{e} model of the
hyperbolic plane, that is the unit disk $\mathbb{D}=\{z:\, |z|<1\}$ supplied with
the metric $\d{\sigma}_\mathbb{D}(z)=|\d{z}|/(1-|z|^2)$, $z\in \mathbb{D}$. 
In this model, the {\it hyperbolic geodesics} are circular arcs that are orthogonal to the unit circle $\mathbb{T}=\{z:\; |z|=1\}$.
A set $S\subset\mathbb{D}$ is {\it hyperbolically convex} if for any two points $z_1$ and $z_2$ in $S$, the hyperbolic geodesic
connecting $z_1$ to $z_2$ lies entirely inside $S$. A {\it hyperbolic $n$-gon} is a simply connected subset of $\mathbb{D}$, 
which contains the origin and which is
bounded by a Jordan curve consisting of $n$ hyperbolic geodesics and arcs of
the unit circle $\mathbb{T}$ which form the {\it sides} of a hyperbolic $n$-gon.
In this paper, the considered $D_n$, $n\geq 3$ is a convex hyperbolic $n$-gon on the unit disk $\mathbb{D}$
having all its vertices $A_1,A_2,\ldots,A_n$ on $\mathbb{T}$ 
and ordered in the positive direction of $\mathbb{T}$ such that 
$$ A_1=1,\quad A_j=\exp\left(2\pi i\sum \limits_{k=1}^{j-1}\alpha_k\right),\quad j=2,3,\ldots,n.$$
where $0< \alpha_k< 1/2$ is the angle corresponding to the side $A_k A_{k+1}$
and having its sides on circles orthogonal to $\mathbb{T}$. By $D_n^*$, we denote 
the regular hyperbolic  $n$-gon on the unit disk $\mathbb{D}$ having all its vertices $A_1^*,
A_2^*,A_3^*,\ldots,A_n^*$ on $\mathbb{T}$ and ordered in the positive direction of $\mathbb{T}$ such that
$$ 
A^*_1=1,\quad A_j^*=\exp \left( 2\pi i\sum \limits_{k=1}^{j-1}\alpha_k^*\right), \quad j=2,3,\ldots,n.
$$
where $\alpha_k^*=1/n$ is the angle corresponding to the side $A_k^*A_{k+1}^*$ and having its 
sides on circles orthogonal to $\mathbb{T}$. 
\par
We consider the following realization model of $D_n$ as a biological cell which replicates itself at a discrete
time $s=0,1,2,\ldots$ (see Figure \ref{modelA}).

\begin{itemize}
\item The cell $D_n^{(0)}=D_n$ is the only cell of generation $0$.  %
\item If we reflect $D_n^{(0)}$ with respect to its sides we get
$n$ new non-overlapping hyperbolic $n$-gons $D_{ns}$, $1\leq s\leq
n$ of generation $1$.  %
\item Now every $D_{ns}$ can be reflected with respect to
each of its $n-1$  ``free'' sides to get new $n-1$ $n$-gons of
generation $2$. Altogether we have $n(n-1)$ $n$-gons of generation $2$.  %
\item Continuing we will have $n(n-1)^2$ $n$-gons of
generation $3$, $n(n-1)^3$ $n$-gons of generation $4$, etc.
\end{itemize}

\begin{figure}%
\centering %
\includegraphics[scale=0.30]{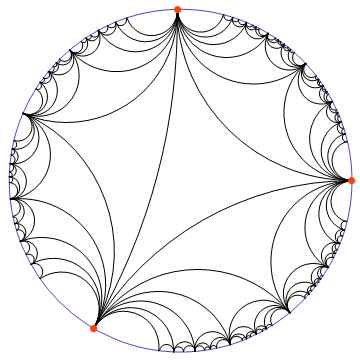}
\caption{Model of cells for $D_3^{(s)}$, $s\geq 0$}
\label{modelA}
\end{figure}

Let $D_n^{(s)}$ be the {\it body} of all generations $\le s$ (i.e.
$D_n^{(s)}$ is again a convex hyperbolic polygon which is precisely the union of 
all convex polygons of generations $j$, $0\le j\le s$). In particular, $D_n^{(1)}$ is a 
convex hyperbolic $n(n-1)$-gon on the unit disk $\mathbb{D}$ 
having all its vertices $B_1,B_2,\ldots,B_{n(n-1)}$ on $\mathbb{T}$
and having its sides on circles orthogonal to $\mathbb{T}$.
Let $\alpha_{j,k}$ be the angle corresponding to the side of $D_n^{(1)}$ obtained by reflecting the 
vertex $A_k$ with respect to the side $A_j A_{j+1}$ of the given $D_n$ where $j=1,2,\ldots,n$. Similarly, by $(D_n^*)^{(1)}$ 
we denote a regular hyperbolic $n(n-1)$-gon on the unit disk $\mathbb{D}$
having all its vertices $B^*_1,B^*_2,\ldots,B^*_{n(n-1)}$ on $\mathbb{T}$ and having its sides on circles orthogonal to $\mathbb{T}$.
Let $\alpha^*_{j,k}$ be the angle corresponding to the side of $(D^*_n)^{(1)}$ obtained by reflecting the 
vertex $A^*_k$ with respect to the side $A^*_j A^*_{j+1}$ of the given $D^*_n$  where $j=1,2,\ldots,n$.
\par
Hyperbolic polygons play not only a significant role in hyperbolic geometry and trigonometry 
but also in the flourishing theory of Fuschian groups, Riemann surfaces and automorphic functions 
\cite{B1,F1,L1}. As a matter of fact, hyperbolic polygons are dense 
in the class of hyperbolically convex regions. Thus on a brighter note, approximation of these regions 
by hyperbolic $n$-gons combined with variational arguments serve as a primary tool to solve
several extremal problems. The regular hyperbolic $n$-gon 
is extremal for many functionals in the hyperbolic plane. {\it Geometrical} results 
include that the regular one maximizes the hyperbolic area among all hyperbolic $n$-gons with a given 
hyperbolic perimeter. Recent results involving functionals of {\it non-geometrical} nature 
such as, eigenvalues of the Laplacian, capacities, conformal radius, harmonic measure etc. 
can be found in \cite{BSH,S1,S2,S3}. 
\par
In this paper, the main result is,

\begin{theorem}\label{thm-main}
Let $\alpha_{j,k}$ and $\alpha^*_{j,k}$ be as defined before for $D_4$.
Then
\begin{equation}\label{eq:1}
\max \limits_{j,k} \alpha_{j,k} \geq \max \limits_{j,k} \alpha^*_{j,k}.
\end{equation}
Equality in (\ref{eq:1}) is attained only in the case when $D_4$ is a regular $4$-gon.
\end{theorem}

The proof of Theorem \ref{thm-main} given in Section $4$ is {\it geometric} and is based 
on the monotonicity results of the inverted sides of the hyperbolic $n$-gon 
discussed in Section $3$. Some well known preliminary results which assist the proof of the results described in the later sections 
are summarized in Section $2$. Section $3$ contains several interesting results related to hyperbolic $n$-gons as 
in the considered model, the inversion of a regular hyperbolic $n$-gon does not always 
produces a regular hyperbolic $n$-gon and the inversion of a non-regular hyperbolic $n$-gon never produce a 
regular hyperbolic $n$-gon. In Section $5$, we discuss a more general geometric problem posed by A.~Solynin.

\section{Preliminaries}
\subsection{Circular Inversion}
Let $\Gamma$ be a fixed circle in the plane with center $O$ and radius $r$. Then the inverse of a point $P$ with respect to 
$\Gamma$ is the point $P'$ lying on the ray from $O$ through $P$ such that $|OP||OP'|=r^2$. Then, we say $P'$ is obtained 
from $P$ by {\it circular inversion} with respect to circle $\Gamma$. This is extensively studied in \cite{H1}. We use the following
two theorems to obtain some interesting results for hyperbolic $n$-gons. The proofs can be found in \cite{H1}, Chapter $7$, Section $37$.

\begin{proposition}
If a circle $\gamma$ is orthogonal to $\Gamma$ {\rm(}at its intersection points{\rm)}, then $\gamma$ is transformed into 
itself by circular inversion in $\Gamma$. Conversely, if a circle $\gamma$ contains a single pair $A$, $A'$ of 
inverse points, then $\gamma$ is orthogonal to $\Gamma$ and is sent into itself. 
\end{proposition}

\begin{proposition}
If $P, P'$ and $Q, Q'$ are pair of inverse points with respect to some
circle with center at $O$ and radius $r$, then 
$$|P'Q'| =  \frac{r^2 |PQ|}{|OP| |OQ|}.$$
\end{proposition}

\subsection{Majorization}

Let ${\bf x},{\bf y}$ be the vectors in $\mathbb{R}^n$. Let $x_{(1)}$ be the largest element in 
${\bf x}$, $x_{(2)}$ be the second largest element, and so on. The vector ${\bf x}$ is said to be 
majorize the vector ${\bf y}$ (denoted ${\bf x} \succ {\bf y}$) if $\sum _{i=1}^k x_{(i)} 
\geq \sum_ {i=1}^k y_{(i)}$ for $k=1,2,\ldots,n-1$ and $\sum _{i=1}^n x_i =\sum _{i=1}^n y_i$.
\par
A real-valued function $\phi$ defined on a set $\mathcal{A}\subset \mathbb{R}^n$ is said to be 
{\it Schur-convex} ({\it Schur-concave}) on $\mathcal{A}$ if
$$ 
{\bf x} \succ {\bf y}\ {\rm on}\ \mathcal{A} \Rightarrow \phi(x)\leq \phi(y)\; (\phi(x)\geq \phi(y)).
$$

A detailed exposition of the properties of majorization and the proof of the following well known theorem by 
Schur, Hardy, Littlewood, P\'{o}lya is given in \cite{MOA}. 
We shall apply Theorem \ref{thm11} to study a more general problem in Section $5$. 

\begin{theorem}\label{thm11}
If $U \subset \mathbb{R}^n$, where $U$ is an open set in $\mathbb{R}^n$, and $g: U \rightarrow \mathbb{R}$ is convex {\rm (}concave{\rm)}, then
$$
\phi({\bf x})= \sum_{i=1}^n g(x_i)
$$
is Schur-convex {\rm(}Schur-concave{\rm )} on $U$, where ${\bf x}=(x_1,x_2,\ldots,x_n)$. Consequently,
$$
 {\bf x} \prec {\bf y}\; ( {\bf x} \succ {\bf y})\ {\rm on}\ U \Rightarrow \phi({\bf x}) \leq \phi({\bf y})
$$
\end{theorem}


\section{General results for hyperbolic $n$-gons}

\begin{figure}%
\centering %
\begin{minipage}{0.5\textwidth}%
\includegraphics[scale=0.4]{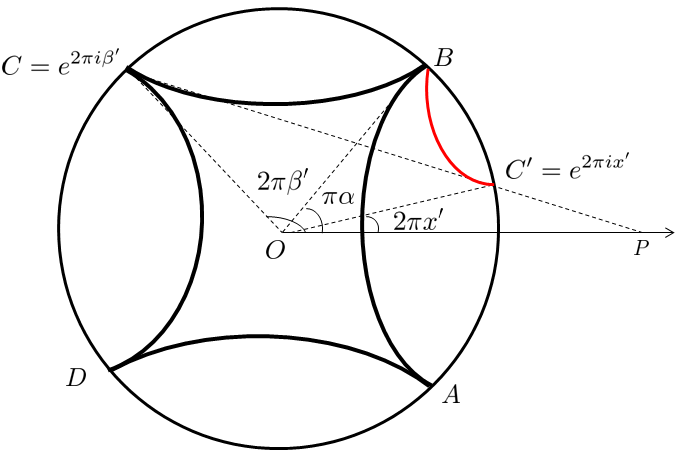}
\end{minipage}
\caption{Formula of the inverse point $C'$}
\label{fig3}
\end{figure}

In this section, we shall first generate a formula of an inverse point with respect to a side of a hyperbolic $n$-gon. 
Suppose $C' :=e^{2\pi i x}$ is the inverse point of $C :=e^{2\pi i\beta}$ with respect to the side $AB$ lying on the circle 
whose center is at P and radius $r$. Define
$A:=e^{2\pi i a}$ and $B:=e^{2\pi i b}$ where $a= \sum _{i=1}^m \alpha_i$ for some fixed $m$ and  
$b=a+\frac{\alpha}{2}$.
Thus, in order to obtain a general formula, we apply the transformation $z \mapsto e^{-2\pi i b}z$ which maps $A$ to $e^{-2\pi i\alpha}$ and 
$B$ to $e^{2\pi i\alpha}$, which symmetrically places $A$ and $B$ with respect to the real axis. To ease our notations, we define
$\beta':=\beta-b$ and $x':=x-b$ for the transformed polygon (See Figure \ref{fig3}). 
Clearly, $OP=\sec \pi \alpha$, $r=\tan \pi \alpha$.
Let us denote the vectors $\overrightarrow{OC}$, $\overrightarrow{OC'}$ by complex numbers $z$ and $z^*$ respectively. Note that,
$z^{*}= e^{2\pi i x'}$ and $z=e^{2\pi i\beta'}$. Then 
$\overrightarrow{OP}+\overrightarrow{PC'}=\overrightarrow{OC'} \Rightarrow \overrightarrow{PC'}=z^{*}-\sec \pi \alpha$
and similarly, $\overrightarrow{PC}=z-\sec \pi \alpha$. Thus, we have $|PC'||PC| = \tan^2 \pi \alpha$. This implies
$\vert z^{*}-\sec \pi \alpha \vert \vert z-\sec \pi \alpha\vert =\tan^2 \pi \alpha$. Thus $\vert z^{*}-\sec\pi \alpha \vert =\tan^2 \pi \alpha/\vert z-\sec \pi\alpha\vert$.
Suppose $\hat{\mu}$ denotes the unit vector in the direction of $\overrightarrow{PC}$ (and so for $\overrightarrow{PC'}$). Therefore, we obtain
\begin{align*}
 e^{2\pi ix'} &=z^{*} = \overrightarrow{OC'}=\overrightarrow{OP}+\overrightarrow{PC'} = \sec \pi \alpha + \hat{\mu} \vert \overrightarrow{PC'} \vert \\
&=\sec \pi \alpha +\frac{z-\sec \pi \alpha}{\vert z-\sec \pi\alpha\vert} \vert z^{*}-\sec \pi \alpha\vert = \sec \pi\alpha +\frac{\tan^2 \pi\alpha}{\overline{z}-\sec \pi\alpha} \\
& = \sec \pi\alpha +\frac{\tan ^2 \pi\alpha}{e^{-2\pi i \beta'}-\sec \pi\alpha}= \frac{1}{\cos \pi\alpha}+\frac{e^{2\pi i\beta'} \sin^2 \pi\alpha}{(\cos \pi\alpha)(\cos \pi\alpha-e^{2\pi i\beta'})}\\
& = \frac{1-e^{2\pi i\beta'}\cos \pi \alpha}{\cos \pi\alpha-e^{2\pi i\beta'}}.\\
\end{align*}

Hence, we obtain
\begin{equation}
\label{eq4}
e^{2\pi ix'}=-e^{2\pi i\beta'} \frac{\cos \pi \alpha-e^{-2\pi i\beta'}}{\cos \pi \alpha-e^{2\pi i\beta'}}.
\end{equation}

Substituting $\beta'$ by $\beta-b$ and $x'$ by $x-b$ in (\ref{eq4}), we obtain
\begin{equation}\label{eq:imp}
e^{2\pi ix}= -e^{2\pi i \beta}\ \frac{\cos \pi\alpha -e^{-2\pi i(\beta-b)}}{\cos \pi \alpha-e^{2\pi i(\beta-b)}}.
\end{equation}
If $\beta-b >0$, considering the argument on both sides of (\ref{eq:imp}), we obtain
\begin{equation}\label{eq1}
2\pi x=-\pi+2\pi \beta+ 2\tan^{-1} \left(\frac{\sin 2\pi (\beta-b)}{\cos \pi\alpha- \cos 2\pi (\beta-b)} \right). 
\end{equation}
If $\beta-b<0$, considering the argument on both sides of (\ref{eq:imp}), we obtain
\begin{equation}\label{eq2}
2\pi x=\pi+2\pi \beta- 2\tan^{-1} \left(\frac{\sin 2\pi (b-\beta)}{\cos \pi \alpha- \cos 2\pi (b-\beta)} \right).
\end{equation}

Thus, the inverse point with respect to a side of a hyperbolic $n$-gon is given by $e^{2\pi ix}$ where $2\pi x$ 
is given as in (\ref{eq1}) and (\ref{eq2}).

\begin{figure}%
\centering %
\begin{minipage}{0.5\textwidth}%
\includegraphics[scale=0.4]{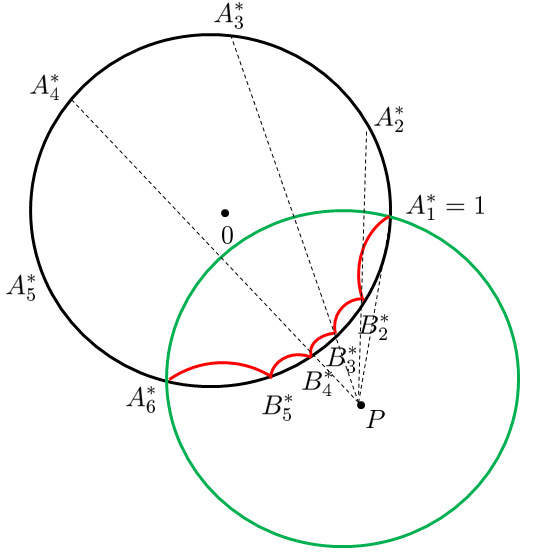}
\end{minipage}
\caption{Monotonicity of inverted sides with respect to one fixed side of $D_6^*$}
\label{fig4}
\end{figure}

\begin{lemma}\label{lem9}
For fixed $j$, where $1\leq j \leq n$, $\alpha_{j,k}^*$ is monotonically decreasing for $k\neq j$ and 
$1\leq k\leq \left \lceil{\frac{n+1}{2}}\right \rceil$ {\rm(}$1\leq k\leq \left \lceil{\frac{n}{2}}\right \rceil${\rm)}, 
if $n$ is even {\rm(}odd{\rm)}. 
\end{lemma}
\begin{proof}
It is sufficient to prove the theorem for $j=n$ and when $n$ is even (see Figure \ref{fig4}). 
We can repeat the arguments to prove the result when $n$ is odd. Let $A_1^*A_n^*$ be the side of $D_n^*$ corresponding to the angle $\alpha_n^*$
which lies on the circle whose center is at P and has radius $r$. Due to symmetry it is sufficient to consider  
$l_1, l_2, \ldots, l_{\left \lceil{\frac{n+1}{2}}\right \rceil}$, the lines joining $P$ to the vertices of 
$D_n^*$, $A^*_1, A^*_2,\ldots, A^*_{\left \lceil{\frac{n+1}{2}}\right \rceil}$ respectively. 
Let $B^*_2, B^*_3,\ldots, B^*_{\left \lceil{\frac{n+1}{2}}\right \rceil}$ be the inverse points of the points
 $A^*_2, A^*_3,\ldots, A^*_{\left \lceil{\frac{n+1}{2}}\right \rceil}$ 
respectively with respect to the side $A^*_1 A^*_n$. Let $m_1, m_2, \ldots, m_{\left \lfloor{\frac{n+1}{2}}\right \rfloor}$ 
denote \linebreak $|A^*_1 B^*_2|, |B^*_2 B^*_3|,\ldots, |B^*_{\left \lfloor{\frac{n+1}{2}}\right \rfloor}B^*_{\left \lceil{\frac{n+1}{2}}\right \rceil}|$ respectively.
 Since the sides of the hyperbolic $n$-gon are orthogonal to $\mathbb{T}$, it is sufficient to show that 
 $m_1>m_2>\ldots> m_{\left \lfloor{\frac{n+1}{2}}\right \rfloor}$. This will further imply that 
 $\alpha^*_{n,2}>\alpha^*_{n,3}>\ldots> \alpha^*_{n,\left \lceil{\frac{n+1}{2}}\right \rceil}$. 
 Since $D^*_n$ is a regular hyperbolic $n$-gon, then $|A^*_1A^*_2|=|A^*_2A^*_3|=\ldots=|A^*_{n-1}A^*_1|$. Also, it is straightforward to note that 
 $l_1< l_2< \ldots <l_{\left \lceil{\frac{n+1}{2}}\right \rceil}$. So, for any $k$, where $1\leq k\leq\left \lfloor{\frac{n+1}{2}}\right \rfloor-1$, we obtain
$$
m_k = \frac{r^2 \vert A^*_k A^*_{k+1} \vert}{l_k l_{k+1}}> \frac{r^2 \vert A^*_{k+1}A^*_{k+2} \vert}{l_{k+1} l_{k+2}}=m_{k+1}.
$$
Hence, the result follows.
\end{proof}

\begin{figure}
    \centering
    \begin{subfigure}[b]{0.45\textwidth}
        \centering
        \includegraphics[width=\textwidth]{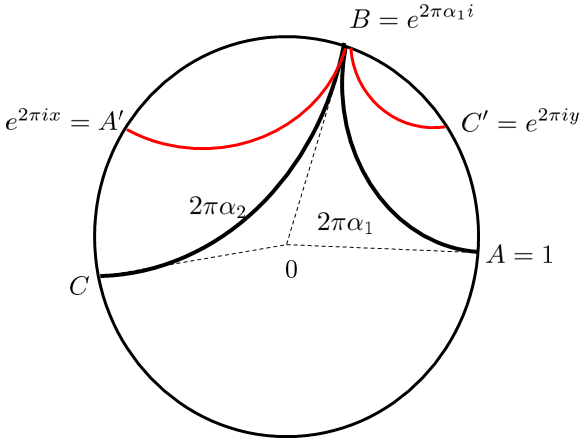}
        \caption{Adjacent sides}
        \label{fig5}
    \end{subfigure}
    \quad
    \begin{subfigure}[b]{0.45\textwidth}
        \centering
        \includegraphics[width=\textwidth]{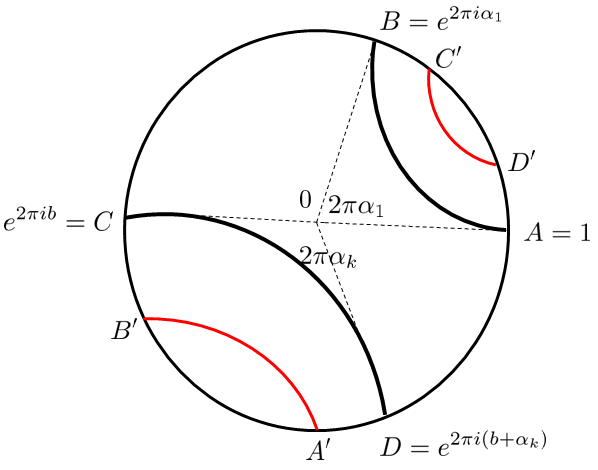}
        \caption{Non-adjacent sides}
        \label{fig6}
    \end{subfigure}
    \caption{Monotonicity of inverted sides in $D_n$}
  \label{fig:three}
\end{figure}

We discuss the monotonicity of the argument of the inverse point and the
monotonicity of the inverted sides with respect to two adjacent and non-adjacent sides of a hyperbolic $n$-gon in Lemma \ref{lem 05}.
This serves as an important tool to prove the main theorem of this paper. 

\begin{lemma}\label{lem 05}
{\rm(i)} {\bf (Monotonicity with respect to one point)} Let $e^{2\pi ix}$ be the inverse point of $e^{2\pi i\beta}$ with respect to the side of $D_n$
whose corresponding angle is $\alpha$. Then $x$ is a decreasing function of $\beta$ for $0 \leq \beta \leq 1$. \\

\noindent {\rm(ii)} {\bf (Monotonicity with respect to two adjacent sides)} Let $\alpha_k$, $\alpha_l$ be the corresponding angles of two adjacent sides of 
$D_n$ and $\alpha_k<\alpha_l$. Then $\alpha_{k,l}<\alpha_{l,k}$.\\

\noindent {\rm(iii)} {\bf (Monotonicity with respect to two non-adjacent sides)} Let $\alpha_k$, $\alpha_l$ for some $k\neq l$ be the 
corresponding angles of two non-adjacent sides of $D_n$ and $\alpha_k<\alpha_l$. Then $\alpha_{k,l}<\alpha_{l,k}$.

\end{lemma}

\begin{proof}
(i) We normalize the initial end point (in the positive direction of $\mathbb{T}$) 
of the side corresponding to $\alpha$ at $1$. Using (\ref{eq1}), for a fixed $\alpha$, the argument of $e^{2\pi ix}$ is given by a 
$2\pi$-multiple of 
$$
f(\beta) := -\pi+2 \pi \beta +2 \tan^{-1} \left( \frac{\sin 2\pi (\beta-\alpha/2)}{\cos \pi \alpha-\cos 2\pi(\beta-\alpha/2)}\right).
$$
We show that $f$ is a decreasing function of $\beta$, where $\alpha\leq \beta <1$. We obtain,
\begin{equation}\label{eq:20}
\frac{\partial f}{\partial \beta} =\frac{-4\pi \sin^2 \pi \alpha}{3-2\cos 2\pi \beta+\cos 2\pi \alpha-2 \cos 2\pi (\beta-\alpha)}. 
\end{equation}
Note that, $f(\alpha)=2\pi\alpha >0$ and $f(1)= 2\pi >0$. Note that the numerator of (\ref{eq:20}) is clearly negative and therefore it 
suffices to show that 
\begin{equation}\label{eq:func}
 g(\beta):=3-2\cos 2\pi\beta+\cos 2\pi\alpha-2 \cos 2\pi(\beta-\alpha) >0. 
 \end{equation}
Clearly $g(\alpha)=1-\cos 2\pi \alpha >0$, $g(1)=1-\cos 2\pi \alpha >0$.
We therefore complete the proof by showing that $g$ is a concave function of $\beta$, where $\alpha \leq \beta<1$. 
Note that, $\frac{\partial g}{\partial \beta}= 2\pi \left[ \sin 2\pi \beta+ \sin 2\pi (\beta-\alpha)\right] $. Then

$$
\frac{\partial g}{\partial \beta}=0  \Rightarrow \sin 2\pi \beta+ \sin 2\pi (\beta-\alpha)=0 \Rightarrow 2 \sin \left( 2\pi \beta- \pi \alpha \right) \cos \pi \alpha =0.  
$$
Either $\sin \left( 2\pi \beta-\pi \alpha \right)=0$ or $\cos \pi \alpha =0$. But the second equality is not possible since 
$0<\alpha<1/2$. Thus $ \sin \left( 2\pi \beta-\pi \alpha \right)=0$ which implies that $ \beta=\frac{\alpha+k}{2}$. 
Since $\alpha<\beta<1$, then $\alpha<k<2-\alpha$. Since, $0<\alpha<1/2$, then $0<k<3/2$, which implies $k=1$. Therefore, 
for a fixed $\alpha$, $(1+\alpha)/2$ is a critical point of $g$. 
Considering a point in $\left( \alpha, \frac{\alpha+1 }{2} \right)$, say $\frac{3\alpha+1}{4}$, we notice that
$\frac{\partial g}{\partial \beta} = 4 \cos \pi \alpha \cos (\pi\alpha/2)>0$. Therefore, $g$ is increasing in $\left( \alpha, \frac{\alpha+1 }{2} \right)$. 
Considering a point in $(\frac{\alpha+1 }{2},1)$, say $\frac{\alpha+3}{4}$, we notice that 
$\frac{\partial g}{\partial \beta}=-4 \cos \pi \alpha \cos (\pi \alpha/2) <0$. Therefore, $g$ is decreasing in $(\frac{\alpha+1 }{2},1)$. 
Hence, $g$ is a concave function of $\beta$ and the result follows.\\

\noindent (ii) As shown in Figure \ref{fig5}, let $\alpha_1$, $\alpha_2$ be the corresponding angles 
(renaming $\alpha_k$, $\alpha_l$ by $\alpha_1$ and $\alpha_2$ respectively) of two adjacent sides 
$AB$ and $CD$ of $D_n$ respectively and $|AB|<|BC|$. We normalize by rotating one end point of the side $AB$ to $1$.
Let $C':=e^{2\pi iy}$, $A':=e^{2\pi ix}$ be the inverse points of $C$ and $A$ with respect to sides 
$AB$ and $BC$ respectively. Using (\ref{eq2}), the argument of $A'$ is,
\begin{equation}
2\pi x=\pi+2\pi \beta_1-2 \tan^{-1} \left( \frac{\sin 2\pi(b_1-\beta_1)}{\cos \pi \alpha_2-\cos 2\pi (b_1-\beta_1)} \right),
\end{equation}
where $\beta_1=0$, $b_1=\alpha_1+\frac{\alpha_2}{2}$. Also, $C'$ is the inverse point of C with respect to the side AB. 
Using (\ref{eq1}), the argument of $C'$ is,
\begin{equation}
2\pi y=-\pi+2\pi \beta_2+2 \tan^{-1} \left( \frac{\sin 2\pi(\beta_2-b_2)}{\cos \pi \alpha_2-\cos 2\pi (b_2-\beta_2)}\right),
\end{equation}
where $\beta_2=\alpha_1+\alpha_2$, $b_2=\frac{\alpha_1}{2}$. Then, $|BC'|$=$2\pi(\alpha_1-y)$, $|A'B|$=$2\pi(x-\alpha_1)$. 
Therefore, we show that $ |BC'|-|A'B|=2\pi (x+y-2\pi \alpha_1) >0$.
Equivalently, it suffices to show that for a fixed $0<\alpha_2<1/2$,
\begin{align*}
 f(\alpha_1):= &-2 \Biggl[ (\alpha_1-\alpha_2)\pi +\tan^{-1} \left( \frac{\sin \pi(2\alpha_1+\alpha_2)}{\cos \pi \alpha_2-\cos \pi(2\alpha_1+\alpha_2)} \right)\\
 &-\tan^{-1} \left( \frac{\sin \pi(\alpha_1+2\alpha_2)}{\cos \pi\alpha_1-\cos \pi(2\alpha_2+\alpha_1)} \right)\Biggr] 
\end{align*}
is a concave function of $\alpha_1$ where $0<\alpha_1<\alpha_2$. Note that, $f(0)=0,\ f(\alpha_2)=0$. We obtain
$$
\frac{\partial f}{\partial \alpha_1}=-4 (\sin^2 \pi \alpha_2) \left( \frac{2}{\Lambda_2}-\frac{1}{\Lambda_1} \right)
= -4 (\sin^2 \pi \alpha_2) \left( \frac{2\Lambda_1-\Lambda_2}{\Lambda_1 \Lambda_2} \right), 
$$
where 
\begin{eqnarray*}
\Lambda_1 &=3+\cos 2\pi\alpha_2-2\cos 2\pi\alpha_1-2\cos 2\pi (\alpha_1+\alpha_2).\\
\Lambda_2 &=3+\cos 2\pi\alpha_1-2\cos 2\pi\alpha_2-2\cos 2\pi (\alpha_1+\alpha_2).
\end{eqnarray*}
It is straightforward to observe that as in (\ref{eq:func}), for fixed $0<\alpha_2<1/2$, $\Lambda_1,\Lambda_2>0$ for $0<\alpha_1<\alpha_2$. 
Also, $\frac{\partial f}{\partial \alpha_1}=0$ implies 
either $\alpha_2=0$ or $2\Lambda_1-\Lambda_2=0$. But $\alpha_2$ can never be zero, therefore, all the points satisfying $2\Lambda_1-\Lambda_2=0$, for a fixed $\alpha_2$, are the critical points of $f$.  
Next, we claim that there exists only one critical point of $f$ in $(0,\alpha_2)$. 
Equivalently, it is sufficient to show that there exists only one zero of the function $g(\alpha_1):=2\Lambda_1-\Lambda_2$ inside $(0,\alpha_2)$ for a fixed $\alpha_2$. 
Note that $ g(0)=-4 \sin^2 \pi\alpha_2 <0$, $ g(\alpha_2)=-4\cos^2 \pi\alpha_2-\cos \pi\alpha_2+5 >0 $.
By the Intermediate Value Theorem, there exists at least one zero inside $(0,\alpha_2)$. To show, there exists only one, we claim that $g(\alpha_1):=2\Lambda_1-\Lambda_2$ is a strictly 
increasing function of $\alpha_1$, where $0<\alpha_1<\alpha_2<\frac{1}{2}$. Consider
$$
 \frac{\partial g}{\partial \alpha_1}= 2\pi \left(5\sin 2\pi \alpha_1+2 \sin 2\pi (\alpha_1+\alpha_2) \right). 
$$
We shall show that $\frac{\partial g}{\partial \alpha_1} >0$. 
We can rewrite it as,
$ 5 \sin \alpha_1+2\sin (\alpha_1+\alpha_2)= \kappa_1 \sin \alpha_1+\kappa_2 \cos \alpha_1 $ where $\kappa_1=5+2\cos \alpha_2$, $\kappa_2=2\sin \alpha_2$. 
Then, using the elementary trigonometry formula, 
\begin{equation*}
 \kappa_1 \sin \alpha_1 +\kappa_2 \cos \alpha_1= \sqrt{\kappa_1^2+\kappa_2^2} \sin (\alpha_1+H), 
 \end{equation*}
where $H=\tan^{-1}\left( \frac{\kappa_2}{\kappa_1} \right)$, it is sufficient to show that $0<\alpha_1+H<\pi$ (since $\sqrt{\kappa_1^2+\kappa_2^2}>0$). 
It is clear that $\alpha_1+H>0$. Since $\alpha_1<\alpha_2$,
$$
 \alpha_1+H=\alpha_1+\tan^{-1} \left(\frac{2\sin \alpha_2}{5+2\cos \alpha_2} \right)<\alpha_2+\tan^{-1} \left(\frac{2\sin \alpha_2}{5+2\cos \alpha_2} \right). 
$$
It can be easily shown that by substituting $\sin \alpha_2$ by $w$ that
$$
 h(w):= \sin^{-1}(w)+\tan^{-1} \left( \frac{2w}{5+2\sqrt{1-w^2}} \right) 
$$
is an increasing function in $(0,1)$. Therefore, for all $w \in (0,1)$,
\begin{align*}
h(w) &< h(1) =\sin^{-1}(1)+\tan^{-1} \left( \frac{2(1)}{5+2(0)} \right)\\
&= \frac{\pi}{2}+\tan^{-1}(2/5)<\frac{\pi}{2}+\frac{\pi}{4}<\pi.\\
\end{align*}
Therefore, $\alpha_1+H<\pi$. Thus, $\frac{\partial f}{\partial \alpha_1}>0$ for $2\Lambda_1-\Lambda_2<0 $ and $\frac{\partial f}{\partial \alpha_1}<0$ for $2\Lambda_1-\Lambda_2>0$. 
Therefore, $f$ is a concave function of $\alpha_1$ which completes the proof.\\

\noindent (iii) As shown in Figure \ref{fig6}, let $\alpha_k$, $\alpha_l$ be the corresponding angles of two non-adjacent sides 
$AB$ and $CD$ of $D_n$ respectively and $|AB|<|BC|$. We normalize one end point of the side AB at $1$.
Notice that the inverse points move continuously on $\mathbb{T}$. Hence by rotating the side $CD$ 
adjacent to $AB$ will reduce this case to (ii) and therefore the result follows.
\end{proof}

To study further properties of the inversion of an hyperbolic $n$-gon, we ask the following questions. 
(a) If we start with a regular polygon with $n$ sides and reflect it with respect to its sides, 
 what can be said about the new polygon? Is it regular or non-regular? 
(b) If we start with a non-regular polygon with $n$ sides and reflect it with respect to its sides, 
what can be said about the new polygon? Is it ever regular or always non-regular? 
The answers to these questions are given by Lemma \ref{lem 04} and \ref{lem 04a}.

\begin{lemma}\label{lem 04}
Let $D_n^{*}$ be a regular hyperbolic polygon with $n \geq 4$ sides. Then $\left( D_n^{*}\right)^{(s)}$ is 
non-regular for $s \geq 1$. Further, $\left( D_3^{*}\right)^{(1)}$ is a regular hyperbolic $6$-gon and
 $\left( D_3^{*}\right)^{(s)}$ is a non-regular hyperbolic polygon for $s \geq 2$.
\end{lemma}

\begin{proof}
If $n=3$, then $(D^*_3)^{(1)}$ is a regular $6$-gon by Lemma \ref{lem 05}(ii). 
Suppose, if possible, $\left( D_n^{*} \right)^{(s)}$ is regular for $s \geq 1$ and for all $n \geq 4$. Then
Lemma \ref{lem 05}(ii) guarantees that $\left( D_n^{*} \right)^{(s-1)}$ is also regular for $s\geq 1$. 
Continuing in a similar fashion, we conclude that $\left( D_n^{*}\right)^{(1)}$ is regular for $n\geq 4$. 
Let $e^{2\pi ix}$ be the inverse point of $e^{2\pi i\beta}$ on the side corresponding to the angle $\alpha$. 
We normalize one endpoint of the side at $1$. Using (\ref{eq1}) for $D^*_n$ with $\alpha=\frac{2\pi}{n}$, 
$\beta=\frac{4\pi}{n}$ and $b=\frac{\pi}{n}$, we have
$$
2\pi x=-\pi+\frac{4\pi}{n}+2\tan^{-1} \left( \frac{\sin 3\pi/n}{\cos \pi/n-\cos 3\pi/n}\right).
$$ 
Then we shall show that $\frac{2\pi}{n}-2\pi x\neq \frac{2\pi}{n(n-1)}$ for $n\geq 4$. Suppose if possible, 
\begin{align*}
&\,\,\,\,\,\,\frac{2\pi}{n}-2\pi x= \frac{2\pi}{n(n-1)}\\
&\Rightarrow \pi-\frac{2\pi}{n-1}=2\tan^{-1} \left( \frac{\sin 3\pi/n}{\cos \pi/n-\cos 3\pi/n}\right)\\
&\Rightarrow \tan \left(\frac{\pi}{2}-\frac{\pi}{n-1} \right)=\frac{\sin 3\pi/n}{\cos \pi/n-\cos 3\pi/n}\\
&\Rightarrow \cos \frac{\pi}{n-1}\cos \frac{\pi}{n}-\left( \cos \frac{\pi}{n-1}\cos \frac{3\pi}{n}+\sin \frac{\pi}{n-1}\sin \frac{3\pi}{n}\right)=0\\
&\Rightarrow \frac{1}{2} \left( \cos \frac{\pi}{n(n-1)} +\cos \frac{(2n-1)\pi}{n(n-1)}-2 \cos \frac{(2n-3)\pi}{n(n-1)}\right)=0\\
&\Rightarrow \frac{1}{2} \left( \cos \frac{\pi}{n(n-1)}-\cos \frac{(2n-3)\pi}{n(n-1)}\right)+\frac{1}{2} \left( \cos \frac{(2n-1)\pi}{n(n-1)}-\cos \frac{(2n-3)\pi}{n(n-1)} \right)=0\\
&\Rightarrow \sin \frac{\pi}{n} \left( \sin \frac{(n-2)\pi}{n(n-1)}-2 \sin \frac{\pi}{n(n-1)} \cos \frac{\pi}{n}\right)=0\\
&\Rightarrow \sin \frac{\pi}{n} \left( 2 \sin \frac{(n-2)\pi}{n(n-1)}-\sin \frac{\pi}{n-1} \right)=0
\end{align*}
Clearly $\sin \frac{\pi}{n}\neq 0$ for $n\geq 4$. So, consider $f(n):=  2 \sin \frac{(n-2)\pi}{n(n-1)}-\sin \frac{\pi}{n-1}$. Clearly, $f(4)=1-\frac{\sqrt{3}}{2}>0$, contradicting our assumption. For $n\geq 5$, $\frac{(n-2)\pi}{n(n-1)}>\frac{\pi}{n+2}$. Therefore,
\begin{align*}
f(n) &> 2 \sin \frac{\pi}{n+2}-\sin \frac{\pi}{n-1}\\
&>2 \left[\frac{\pi}{n+2}-\frac{1}{3!} \left( \frac{\pi}{n+2}\right)^3 \right]-\frac{\pi}{n-1}\\
&= \frac{\pi}{3(n-1)(n+2)^3} \left[3n^3-\pi^2n^2-(\pi^2+36)n+2(\pi^2-24)\right].
\end{align*}
A straightforward calculus argument will suggest that the above expression is strictly positive for $n\geq 5$, which is a contradiction. Thus, $\left( D_n^{*}\right)^{(1)}$ is not regular for $n\geq 4$.
Similar arguments will show that $\left( D_3^{*} \right)^{(s)}$, $s\geq 2$ is non-regular.
\end{proof}

\begin{lemma}\label{lem 04a}
Let $D_n$ be a non-regular hyperbolic polygon with $n \geq 3$ sides. Then $\left( D_n \right)^{(s)}$ 
is non-regular for $s \geq 1$ and for all $n \geq 3$. 
\end{lemma}
\begin{proof}
Suppose, if possible, $\left( D_n \right)^{(s)}$ is a regular hyperbolic $n$-gon with
 $s\geq 1, n>3$. Then Lemma \ref{lem 05}(ii) guarantees that  $\left( D_n \right)^{(s-1)}$ is a regular polygon. 
However Lemma \ref{lem 04} confirms that the inversion
of a regular polygon always results in a non-regular polygon, which contradicts our hypothesis. 
Hence, the result follows. Similar arguments will conclude that $\left( D_3 \right)^{(s)}$ 
is non-regular for $s\geq 2$. Also,  $\left( D_3 \right)^{(1)}$ is non-regular by Lemma \ref{lem 05}(ii).
\end{proof}


\section{Proof of Theorem \ref{thm-main}}

In this section, we illustrate the proof of Theorem \ref{thm-main} by using variation of vertices of $D_4$ or their continuous movement
on the unit circle $\mathbb{T}$ to exhaust the possible configurations to the extremal one, which is $D_4^*$ as shown in Figure \ref{extremal} where the longest inverted sides of $D_4^*$ are the eight corner ones due to Lemma \ref{lem9}.
We start with an assumed extremal configuration with a fixed number of longest inverted sides. 
The movement of vertices results in change of side-lengths of $D_4$ but preserves the number of sides to produce a ``new'' $D_4$.
We vary the vertices in such a way that causes an increment in the side-length of the ``new'' $D_4$ corresponding to an 
inverted non-longest side in $D_4^{(1)}$ and thus a decrement in the side-length corresponding to 
the inverted longest side. We next argue that there are only two possible movements of the vertices which cause the increment/decrement which further
 contradicts the assumed extremal configuration due to Lemma \ref{lem 05}(ii) and (iii). These concurrently guarantee that there cannot be only 
 one longest side in the extremal configuration.
\par
Suppose first that there is only one longest side in the extremal configuration. The possible configurations are listed in Figure \ref{One extremal side configuration}. 
Due to the symmetry of the $4$-gon the remaining cases of one longest side configuration will be the same as what we discuss here. 
Suppose $\alpha_{1,4}$ is the angle corresponding to the longest side (see Figure \ref{fig7}). Then by Lemma \ref{lem 05}(ii), $\alpha_1 >\alpha_4$.
We move A continuously such that $\alpha_4 >\alpha_1$. This results in $\alpha_{4,1} >\alpha_{1,4}$ (by Lemma \ref{lem 05}(ii)), which gives a contradiction. 
Suppose $\alpha_{1,3}$ is the angle corresponding to the longest side (see Figure \ref{fig8}). Then by Lemma \ref{lem 05}(iii), $\alpha_1 >\alpha_3$. 
We move C continuously such that $\alpha_3>\alpha_1$. This results in $\alpha_{3,1} >\alpha_{1,3}$ (by Lemma \ref{lem 05}(iii)), which again gives a contradiction. Thus, there are more than one longest side in the extremal configuration. 
The following Lemmas \ref{lem-adj} and \ref{lem-opp} explain two particular configurations when the adjacent and opposite sides of extremal $D_4$ 
are equal respectively. 

\begin{figure}
    \centering
    \begin{subfigure}[b]{0.40\textwidth}
        \centering
        \includegraphics[width=\textwidth]{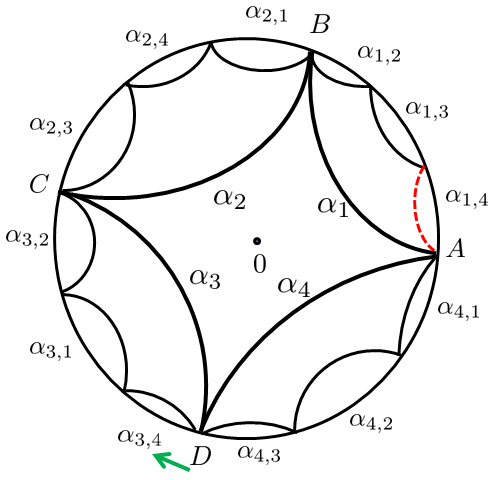}
        \caption{}
        \label{fig7}
    \end{subfigure}
    \quad
    \begin{subfigure}[b]{0.40\textwidth}
        \centering
        \includegraphics[width=\textwidth]{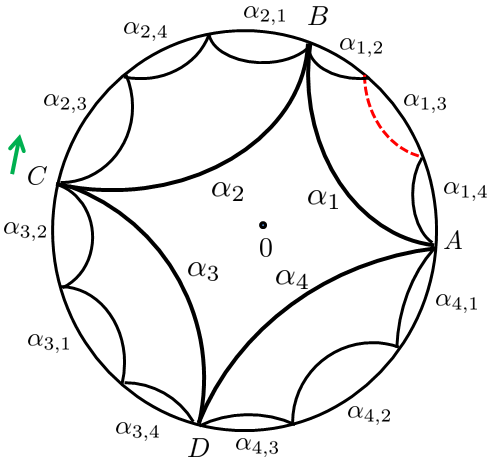}
        \caption{}
        \label{fig8}
    \end{subfigure}
		\caption{Variation in one longest side extremal configuration}
    \label{One extremal side configuration}
\end{figure}

\begin{lemma}\label{lem-adj}
Let $D_4$ be an extremal hyperbolic polygon such that pair of adjacent sides are equal 
 and has at least two adjacent corner longest sides after inversion. Then $D_4=D_4^*$.
\end{lemma}

\begin{proof}
Without loss of generality, suppose $\alpha_1>\frac{\pi}{2}$. Given that $\alpha_1=\alpha_4$ and $\alpha_2=\alpha_3$ which forces $\alpha_{1,4}=\alpha_{4,1}$ and $\alpha_{2,3}=\alpha_{3,2}$. 
Suppose $\alpha_{1,4}$ and $\alpha_{4,1}$ are the angles corresponding to the longest sides in the extremal configuration (see Figure \ref{fig-adj}). 
 By Lemma \ref{lem 05}(i), $x$ is a decreasing function of $\alpha_1$, 
where $e^{2\pi ix}$ is the inverse point of $e^{2\pi i\alpha_1}$. As $\alpha_1$ decreases, $-\alpha_1$ increases. 
Therefore, $\alpha_{1,4}<\alpha_{2,3}$ and $\alpha_{4,1}<\alpha_{3,2}$, which is a contradiction. Thus, all the sides are equal and thus $D_4^*$ gives the extremal configuration.
\end{proof}

\begin{figure}
    \centering
    \begin{subfigure}[b]{0.40\textwidth}
        \centering
        \includegraphics[width=\textwidth]{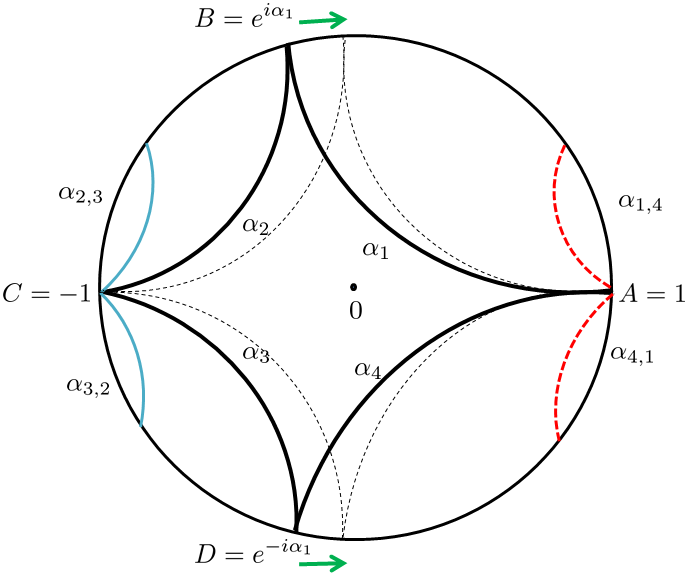}
        \caption{Adjacent longest sides}
        \label{fig-adj}
    \end{subfigure}
    \quad
    \begin{subfigure}[b]{0.40\textwidth}
        \centering
        \includegraphics[width=\textwidth]{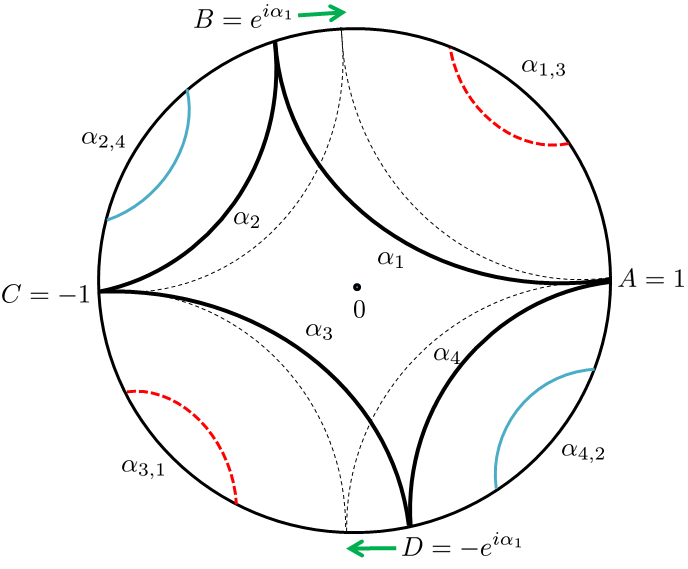}
        \caption{Opposite longest sides}
        \label{fig-opp}
    \end{subfigure}
    \caption{Vertex movement for adjacent and opposite longest sides}
 \label{fig:graphs}
\end{figure}

\begin{figure}
    \centering
   
    \begin{subfigure}[b]{0.40\textwidth}
        \centering
        \includegraphics[width=\textwidth]{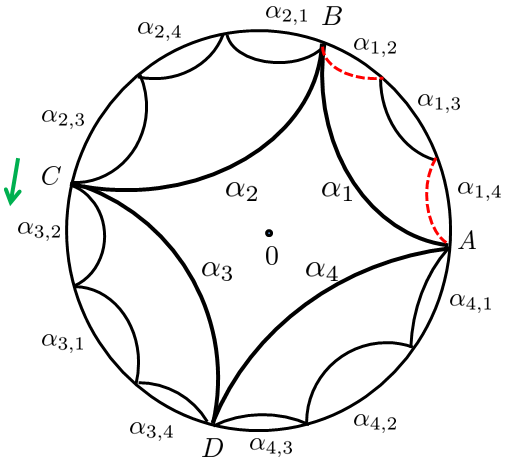}
        \caption{Free-vertex case}
        \label{fig10}
    \end{subfigure}
		\quad 
	\begin{subfigure}[b]{0.40\textwidth}
        \centering
        \includegraphics[width=\textwidth]{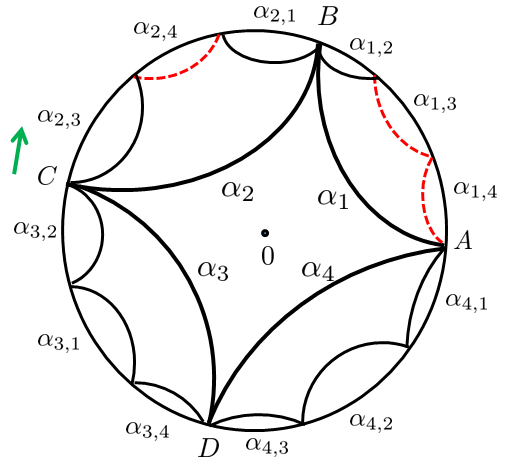}
        \caption{Non-free-vertex case}
        \label{fig19}
    \end{subfigure}
    \caption{Two types of variation}
    \label{Extremal sides configuration}
\end{figure}

\begin{lemma}\label{lem-opp}
Let $D_4$ be an extremal hyperbolic polygon such that pair of the opposite sides are equal 
and has at least two opposite longest sides after inversion. Then $D_4=D_4^*$.
\end{lemma}

\begin{proof}
Given that $\alpha_1=\alpha_3$ and $\alpha_2=\alpha_4$. This forces $\alpha_{1,3}=\alpha_{3,1}$ and $\alpha_{2,4}=\alpha_{4,2}$. 
Suppose $\alpha_{1,3}$ and $\alpha_{3,1}$ are the angles corresponding to longest sides in the extremal configuration (see Figure \ref{fig-opp}). 
By Lemma \ref{lem 05}(i), if $\alpha_1$ decreases then $\pi+\alpha_1$ decreases and therefore $\alpha_{1,3}<\alpha_{2,4}$ and $\alpha_{3,1}<\alpha_{4,2}$, which is a contradiction. 
Thus, all the sides are equal which forces $D_4=D_4^*$.
\end{proof}

We further describe here the remaining two types of variations which demonstrates the exhaustion of all possible configurations that 
can be considered to reach the conclusion that the extremal configuration is when $D_4=D_4^*$.
The analysis of variation described here works identically for any number of longest sides concerned in the assumed extremal configuration. \\
\noindent
(a) The first type of variation is when we vary a ``free'' vertex in a sense that it affects only one inverted longest side in $D_4^{(1)}$ and the others remain unaffected. This variation is already shown in Figure \ref{One extremal side configuration}.
To get a better view of this case for more number of longest sides, suppose that $\alpha_{1,4}$ and $\alpha_{1,2}$ are the angles corresponding to longest sides 
(see Figure \ref{fig10}). 
Then by Lemma \ref{lem 05}(ii), $\alpha_1 >\alpha_2$. We move C such that $\alpha_2>\alpha_1$. 
This results in $\alpha_{2,1} >\alpha_{1,2}$ (by Lemma \ref{lem 05}(ii)), 
a contradiction to the assumption. Thus $\alpha_{1,2}=\alpha_{2,1}$. Thus by Lemma \ref{lem-adj}, the 
extremal configuration is $D_4^*$ as shown in Figure \ref{extremal}.\\
\noindent
(b) The second type of variation is when we vary a ``non-free'' vertex in a sense that it affects some or all inverted longest sides in 
$D_4^{(1)}$. For the better understanding of this case, suppose that $\alpha_{1,4}$, $\alpha_{1,3}$, $\alpha_{2,4}$ are the angles corresponding to longest sides (see Figure \ref{fig19}). Then by Lemma \ref{lem 05}(ii), (iii) 
$\alpha_1$, $\alpha_2 >\alpha_3$, $\alpha_4$. We move C such that $\max \{ \alpha_1, \alpha_2 \}<\min \{\alpha_3,\alpha_4 \}$ which results in 
$\alpha_{4,2} >\alpha_{2,4}$ and $\alpha_{3,1}>\alpha_{1,3}$ (by Lemma \ref{lem 05}(ii), (iii)). 
This is however a contradiction. Thus, the only possibility is 
when $\alpha_{1,3}=\alpha_{3,1}$ and $\alpha_{2,4}=\alpha_{4,2}$. Therefore, by Lemma \ref{lem-opp}, the extremal configuration is $D_4^*$ as shown in Figure \ref{extremal}.\\

\section{Discussion}

\begin{figure}%
\centering %
\begin{minipage}{0.45\textwidth}%
\includegraphics[scale=0.30]{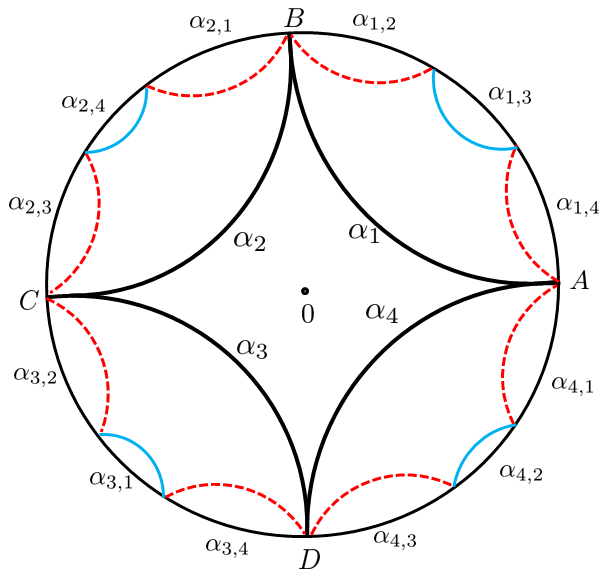}
\end{minipage}
\caption{Extremal Configuration}
\label{extremal}
\end{figure}

Consider a hyperbolic $n$-gon $D_n$ discussed in Section $1$. Since all the vertices of $D_n$ are on $\mathbb{T}$,
then it is well know that the maximal hyperbolic area of $D_n$ is $\frac{\pi}{4}(n-2)$. With an aim to find the explicit 
formula for the Euclidean area of $D_n$, let $\alpha$ be the angle corresponding to
one such side of $D_n$. Then, as shown in Figure \ref{fig1}, 
it is sufficient to find the area of the region $OAD$. Using $\triangle OAB$, we obtain $r= \tan (\pi {\alpha})$.
Therefore, the area of the sector $ABC$ is equal to 
$r^2\theta/2=(\tan ^2(\pi \alpha)/2) ( \pi/2-\pi \alpha)$ 
and the area of $\triangle OAB= \frac{1}{2} |OA||OB|= \frac{1}{2} \tan(\pi \alpha)$. Thus, the 
area of the sector $OCA$
$$ \frac{1}{2} \tan (\pi \alpha) \left( 1-\tan (\pi \alpha)\left( \frac{\pi}{2}-\pi \alpha \right) \right).$$

Therefore, the area of the required region $OAD$ is given by 
$$ F(\alpha):= \tan (\pi \alpha) \left( 1-\tan (\pi \alpha)\left( \frac{\pi}{2}-\pi \alpha \right)\right).$$
Let $\alpha_1,\alpha_2,\ldots,\alpha_n$ be the angles corresponding to the sides of $D_n$. Then the Euclidean area of $D_n$ is given by
\begin{equation}\label{eq:5}
\sum _{k=1}^n F({\alpha}_k)= \sum _{k=1}^n \tan \pi{\alpha}_k \left[ 1- \pi \tan \pi {\alpha}_k \left(\frac{1}{2} -{\alpha}_k \right) \right]. 
\end{equation}

\begin{lemma}\label{lem 01}
For $\alpha \in \mathbb{R}$,
\be \label{eq:F}
F(\alpha) = \tan \pi\alpha \left[ 1-\pi \tan \pi \alpha \left( \frac{1}{2} -\alpha \right) \right]
\ee
is a concave function of $\alpha$ for $0<\alpha <\frac{1}{2}$.
\end{lemma}
\begin{proof}
To show that $F(\alpha)$ is a concave function of $\alpha$, it is sufficient to show that 
$\frac{\d{F}^2}{\d \alpha^2} <0$. We have,
$$ 
\frac{\d{F}^2}{\d \alpha^2} = -\pi^2 \sec^4 \pi\alpha \left[ \pi (1-2\alpha)(2-\cos 2\pi \alpha)-3 \sin 2\pi \alpha\right]. 
$$
Again it suffices to show that 
$$ 
g(\alpha) := \pi (1-2\alpha)(2-\cos 2\pi \alpha)-3 \sin 2\pi \alpha 
$$
is a strictly decreasing function of $\alpha$, which in turn proves that $g(\alpha) > g\left( \frac{1}{2} \right) =0$.
Notice that
\begin{align*}
\frac{dg}{d\alpha} &= 2\pi \left[ -2-2 \cos 2\pi\alpha +\pi(1-2\alpha)\sin 2\pi\alpha \right] \\
&\leq 2\pi (-2-2\cos 2\pi\alpha ) \\
&=-2\pi \sin^2 \pi\alpha<0.
\end{align*}
in $(0,1/2)$. Therefore, the result follows.
\end{proof}

\begin{figure}%
\centering %
\begin{minipage}{0.5\textwidth}%
\includegraphics[scale=0.5]{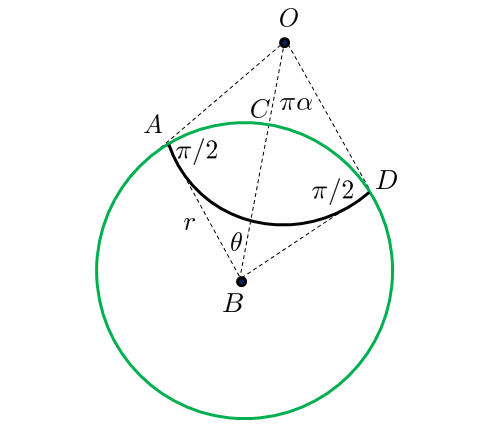}
\end{minipage}
\caption{Area of $D_n$}
\label{fig1}
\end{figure}

\begin{theorem}\label{thm 2}
Let $D_n ^*$ and $D_n$, $n\geq 3$ be as defined before. Then 
\begin{equation}\label{eq:rh}
\Area (D_n) \leq n \tan \frac{\pi}{n} \left[ 1-\frac{\pi (n-2)}{2n} \tan \frac{\pi}{n} \right], 
\end{equation}
where $\Area(D_n)$ is the Euclidean area of $D_n$. Equality in (\ref{eq:rh}) is attained only if $D_n$ is a rotation of $D_n^*$ about the origin.
\end{theorem}

\begin{proof}
We maximize the area of $D_n$ given in (\ref{eq:5}) subject to the condition $\alpha_1 +\alpha_2 +\dots+\alpha_n=1$.
By Lemma \ref{lem 01}, 
for each $0 <\alpha_k <1/2$, $F(\alpha_k)$ is a concave function of $\alpha_k$.
Thus using Jensen's inequality \cite{MOA}, we obtain
$$
F\left( \frac{\sum \limits_{k=1}^n \alpha_k}{n} \right) \geq \frac{\sum \limits_{k=1}^n F(\alpha_k)}{n} 
\Rightarrow \sum \limits_{k=1}^n F(\alpha_k) \leq n F\left( \frac{1}{n}\right).
$$
where $F(1/n)$ is computed using (\ref{eq:F}) to get the right hand side expresssion in (\ref{eq:rh}).
\end{proof}

A.~Solynin suggested an immediate {\it geometric} question ``{\it How does the Euclidean area of $D_n^{(s)}$ grow for $s\geq 0$}?''
In particular,  
\begin{conjecture}\label{con1}
Let $D_n^{(s)}$ and $(D^*_n)^{(s)}$, $s\geq 0$ and $n\geq 3$ be as defined before. Then
\begin{equation}\label{eq-con1}
\Area (D_n^{(s)}) \leq \Area((D_n^*)^{(s)}) .
\end{equation}
with the sign of equality only if $D_n$ is a rotation of $D_n^*$ about the origin.
\end{conjecture}

The case $s=0$ for any $n\geq 3$ of Conjecture \ref{con1} is proved in Theorem \ref{thm 2} that the 
Euclidean area for $D_n$ is maximal for the regular $n$-gon $D^*_n$. 
Also, by Lemma \ref{lem 04}, it is known that $(D_3^*)^{(1)}$ is a regular $6$-gon and hence the conjecture \ref{con1} is proved for $s=1$ and $n=3$.
 However, it is not straightforward to prove the result for any $n\geq 4$ and $s\geq 1$. It is clear that the method applied in 
 Theorem \ref{thm 2} does not work for $s\geq 1$ and $n\geq 4$ due to the complexity of the problem and in particular the 
involvement of a large number of sides.
 In support of the inequality conjectured in (\ref{eq-con1}) we discuss here the application of majorization techniques discussed in Section $2$ which may serve as a 
 promising line of attack to prove Conjecture \ref{con1} for $s=1$ and for any $n\geq4$.  
\par
Let $\mathcal{B}:=\left( \beta_{(1)},\beta_{(2)},\ldots,\beta_{(n(n-1))} \right)$ be a decreasing 
rearrangement of $\lbrace \alpha_{j,k}\rbrace$ and $\mathcal{B}^*:=(\beta_{(1)}^*,\beta_{(2)}^*,\ldots,\beta_{(n(n-1))}^*)$ be a 
decreasing rearrangement of $\lbrace \alpha_{j,k}^*\rbrace$, that is,
$\beta_{(1)} = \max _{j,k} \lbrace \alpha_{j,k}\rbrace$, $\beta_{(2)} =$ second largest 
$\alpha_{j,k},\ldots, \beta_{(n(n-1))} = \min_{j,k} \lbrace \alpha_{j,k}\rbrace$ and so on for $\mathcal{B}^*$.
 Notice that, to prove Conjecture \ref{con1} for $s=1$, it is sufficient to prove the following conjecture:

\begin{conjecture}\label{con2}
Let $\mathcal{B}$ and $\mathcal{B}^*$ be as defined before. Then $\mathcal{B} \succ \mathcal{B}^*$.
\end{conjecture}

Notice that to prove Conjecture \ref{con2}, we need to show 
\begin{equation}\label{eq-proof}
\sum_{i=1}^m \beta_{(i)} \geq \sum_{i=1}^m \beta^*_{(i)}
\end{equation} 
for all $m=1,2,\ldots, n(n-1)$. In particular, we already proved in Theorem \ref{thm-main} that $\beta_{(1)} \geq \beta_{(1)}^*$ for $D_4$. 
We strongly believe that repeating similar arguments, one can generalize the result of Theorem \ref{thm-main} for any $D_n$.
However, we need a different tool to prove the remaining cases of (\ref{eq-proof}). 
By substituting $\alpha$ by $\alpha_{j,k}$ in Lemma \ref{lem 01}, we obtain the function $F(\alpha_{j,k})$. 
So, to prove Conjecture \ref{con1} for $D_n^{(1)}$, we maximize {\it area}$(D_n^{(1)})$ which is,
$$
G(\mathcal{B}) := \sum \limits_{j=1}^{n(n-1)} F(\beta_{(j)}) 
$$
subject to $\sum \limits_{k=1}^{n-1} \alpha_{j,k} =\alpha _j$, $j=1,2,,\ldots,n$ and $\sum \limits_{j=1}^n \alpha_j =1$.
Lemma \ref{lem 01} affirms that $G(\mathcal{B})$ is a concave function and Theorem \ref{thm11} confirms that $G$ is a Schur-concave function. Therefore, 
$G(\mathcal{B}) \leq G(\mathcal{B}^*)$ which supports the inequality conjectured in Conjecture \ref{con2} for $s=1$ is true.\\

\noindent 
{\bf Acknowledgements.} I would like to thank A.~Solynin for introducing the problem and for helpful discussions.


%


\begin{thebibliography}{6.8in}%

\bibitem{BSH}  R.~W.~Barnard, P.~Hadjicostas, A.~Yu.~Solynin, {\it The Poincar\'{e} metric and isoperimetric inequalities for hyperbolic polygons}, Trans. Amer. Math. Soc. \textbf{357} (2005), 3905-3932.

 
\bibitem{B1} A.~F.~Beardon, {\it The geometry of discrete groups},  Graduate Texts in Mathematics, 91, Springer-verlag, New York, (1995).

  
\bibitem{F1} L.~R.Ford, {\it Automorphic functions}, 2nd Ed. Chelsea, New York, (1951).


\bibitem{H1} R.~Hartshorne, {\it Geometry: Euclid and Beyond}, Springer (2000).

\bibitem{H2} J.~Hersch,  {\it On the reflection principle and some elementary ratios of conformal radii}, J. Analyse Math. \textbf{44} (1984/85), 251 -268.

\bibitem{L1} O.~Lehto, {\it Univalent functions and Teichm\"{u}ller spaces}, Springer-verlag, New York, (1987).

\bibitem{MOA}  A.~W.~Marshall, I.~Olkin, B.~C.~Arnold, {\it Inequalities: Theory of Majorization and Its Applications}, Academic Press, Inc. (1979).

\bibitem{S1} A.~Solynin, {\it Some extremal problems on the hyperbolic polygons. (English summary)}, Complex Variables Theory Appl. \textbf{36} (1998), 207-231. 

\bibitem{S2} A.~Yu.~Solynin, {\it Some extremal problems on circular polygons}, J. Math. Sci. \textbf{80} (1996), 1956-1961.

\bibitem{S3}  A.~Yu.~Solynin, V.~A.~Zalgaller, {\it An isoperimetric inequality for logarithmic capacity of polygons}, Ann. of Math. \textbf{159} (2004), 277-303. 
    


\end{thebibliography}
\end{document}